\documentclass[12pt]{article}
\usepackage[affil-it]{authblk}

\usepackage{amsthm}
\usepackage{amssymb}
\usepackage{amsmath}
\usepackage{amsfonts}
\usepackage{cancel}
\usepackage{times}
\usepackage{bbm}
\usepackage{cite}
\usepackage{color, soul}
\usepackage{hyperref}
\hypersetup{
     colorlinks   = true,
     citecolor    = black,
     linkcolor    = blue
}

\usepackage[paper=a4paper, top=2.5cm, bottom=2.5cm, left=2.5cm, right=2.5cm]{geometry}

\def\qed{\hfill $\vcenter{\hrule height .3mm
\hbox {\vrule width .3mm height 2.1mm \kern 2mm \vrule width .3mm
height 2.1mm} \hrule height .3mm}$ \bigskip}

\def \RR {\mathbb R}
\def \NN {\mathbb N}
\def \EE {\mathbb E}

\def \eps {\varepsilon}
\def \vphi {\varphi}

\newtheorem{theorem}{Theorem}
\newtheorem{lemma}{Lemma}

\newtheorem{corollary}[theorem]{Corollary}
\theoremstyle{definition}

\theoremstyle{remark}
\newtheorem{remark}[theorem]{Remark}
\newcommand{\cc}{\kappa}
\newcommand{\s}{D}
\newcommand\support[1]{\mathrm{diam}(\mathrm{supp}(#1))}
\newcommand{\Id}{\mathrm{I}}
\long\def\symbolfootnotetext[#1]#2{\begingroup
\def\thefootnote{\fnsymbol{footnote}}\footnotetext[#1]{#2}\endgroup}
\newcommand{\T}{\mathsf{T}}
\newcommand{\var}{\mathrm{Var}}
\newcommand{\dens}{f}
\newcommand{\bdlc}{e^{\frac{1 - \cc\frac{\s^2}{4}}{2}} \frac{\s}{2}}
\newcommand{\bdlcs}{e^{1 - \cc\frac{\s^2}{4}} \frac{\s^2}{4}}
\newcommand{\diff}{\nabla}
\newcommand{\conv}{\beta}

\newcommand{\optt}{\varphi^{\mathrm{opt}}}
\newcommand{\kmt}{\varphi^{\mathrm{flow}}}
\newcommand{\cov}{\operatorname{Cov}}
\begin{document}
		\title{On the Lipschitz properties of transportation along heat flows}
		
		\author{Dan Mikulincer\thanks{Department of Mathematics, Massachusetts Institute of Technology, Cambridge, MA, USA\\Email address: \texttt{danmiku@mit.edu}}~~and Yair Shenfeld\thanks{Department of Mathematics, Massachusetts Institute of Technology, Cambridge, MA, USA\\Email address: \texttt{shenfeld@mit.edu}} }
		
		
		\date{}
		\maketitle
		\begin{abstract}
			We prove new Lipschitz properties for transport maps along heat
			flows, constructed by Kim and Milman. For (semi)-log-concave measures
			and Gaussian mixtures, our bounds have several applications: eigenvalues comparisons, dimensional functional inequalities, and domination of distribution functions.
		\end{abstract}

\section{Introduction and main results}
In recent years, the study of Lipschitz transport maps has emerged as an important line of research, with applications in probability and functional analysis. Let us fix a measure $\mu$ on $\RR^d$. It is often desirable to write $\mu$ as a
push-forward $\mu=\varphi_*\eta$, for a well-behaved measure $\eta$ and a Lipschitz map $\varphi:\RR^d \to \RR^d$. The main advantage of this approach lies in the fact that one can use the regularity of $\varphi$ to transfer known analytic properties from $\eta$ to $\mu$, compensating for the potential complexity of $\mu$. 

Perhaps the most well-known result in this direction is due to Caffarelli \cite{caffarelli2000monotonicity}, which states that if $\gamma_d$ is the standard Gaussian in $\RR^d$, and $\mu$ is more log-concave than $\gamma_d$, then there exists a $1$-Lipschitz map $\optt$ such that $\optt_*\gamma_d = \mu$. The map $\optt$ is known as the \emph{optimal transport map} \cite{brenier1991polar}. Crucially, the Lipschitz constant of $\optt$ does not depend on the dimension $d$ and, consequently, $\optt$ transfers functional inequalities from $\gamma_d$ to $\mu$, in a dimension-free fashion. For example, the optimal bounds on the
Poincar\'e and log-Sobolev constants are recovered for the class of strongly log-concave measures \cite{cordero2002some}. The main goal of this work is to establish quantitative generalizations of this fact, for measures that are not necessarily strongly log-concave. To this end, we shall use a different transport map, $\kmt$, along the heat flow, of Kim and Milman \cite{kim2012generalization}, which was previously used, in the context of functional inequalities, by Otto and Villani \cite{otto2000generalization}.\footnote{In general, the maps $\kmt$ and $\optt$ are not the same, see \cite{tanana2021comparison}.}

In general, there is no reason to expect that an arbitrary measure could be represented as a push-forward of $\gamma_d$ by a Lipschitz map. Indeed, in line with the above discussion, such measures must satisfy certain functional inequalities with constants that are determined by the regularity of the mapping.
Thus, we restrict our attention to classes of measures that contain, among others, log-concave measures with bounded support and Gaussian mixtures.

We now turn to discuss, in greater detail, the types of measures for which our results shall hold.
First, we consider log-concave measures with support contained in a ball of radius $\s$. It is a classical fact that these measures satisfy Poincar\'e \cite{payne1960optimal} and log-Sobolev \cite{frieze1999sobolev} inequalities with constants of order $\s$. For this reason, Kolesnikov raised the question of whether, in this setting, the optimal transport map $\optt$ is $O(\s)$-Lipschitz \cite[Problem 4.3]{kolesnikov2011mass}. Up to now, the best known estimate, in \cite[Theorem 4.2]{kolesnikov2011mass}, gave a Lipschitz constant that is of order $\sqrt{d}\s$. One of our main contributions is to close this gap, for the map $\kmt$. In fact, we prove a stronger result that captures a trade-off between the convexity of $\mu$ and the size of its support. 

In the sequel, for $\cc \in \RR$ (possibly negative), we say that $\mu$ is $\cc$-log-concave if its support is convex, its density is smooth, and, for $\mu$-almost every $x$, its density satisfies,
$$-\nabla^2\log \left(\frac{d\mu}{dx}(x)\right) \succeq \cc\Id_d.$$

\begin{theorem} \label{thm:lc}
	Let $\mu$ be a  $\cc$-log-concave probability measure on $\RR^d$, and set $\s := \support{\mu}$. Then, for the map $\kmt:\RR^d \to \RR^d$, which satisfies $\kmt_*\gamma_d = \mu$, the following holds:
	\begin{enumerate}
		\item If $\cc > 0$ then,
		$$\|\diff\kmt(x)\|_\mathrm{op}\leq \frac{1}{\sqrt{\cc}},$$
		for $\gamma_d $-almost every $x$.

		\item If $\cc\s^2 < 4$ then,
		$$\|\diff\kmt(x)\|_\mathrm{op}\leq e^{\frac{1 - \cc\frac{\s^2}{4}}{2}} \frac{\s}{2},$$
		for $\gamma_d $-almost every $x$.
	\end{enumerate}
\end{theorem}

Item 1 of Theorem \ref{thm:lc} follows from the result of Kim and Milman \cite{kim2012generalization}, and is analogous to Caffarelli's result \cite{caffarelli2000monotonicity}. Item 2 improves and generalizes
the bound in Item 1 in two ways:
\begin{itemize}
	\item When $\cc > 0$ and $\cc\s^2 <4$, since $e^{\frac{1 - \cc\frac{\s^2}{4}}{2}} \frac{\s}{2}< \frac{1}{\sqrt{\cc}}$, Item 2 offers a strict improvement of the Lipschitz constant in Caffarelli's result.
	
	\item When $\cc \leq 0$, Theorem \ref{thm:lc} provides a Lipschitz transport map for measures that are not strongly log-concave. In particular, the case $\kappa = 0$ is precisely the setting of Kolesnikov's question \cite[Problem 4.3]{kolesnikov2011mass}.
\end{itemize}

Theorem \ref{thm:lc} may also be compared with \cite[Theorem 1.1]{colombo2017lipschitz}, which studies Lipschitz properties of the optimal transport map when the target measure is a semi-log-concave perturbation of $\gamma_d$. We point out that the two results apply in different regimes: while our result applies to semi-log-concave measures with bounded support, the result of \cite{colombo2017lipschitz} requires that the support of the measure is the entire $\RR^d$.\\
 
 The other type of measures we consider are Gaussian mixtures of the form $\mu = \gamma_d \star \nu$, where $\nu$ has bounded support. It was recently shown that these measures satisfy several dimension-free functional inequalities \cite{chen2021dimension,bardet2018functional,wang2016functional}. As we shall show, this phenomenon can be better understood and further strengthened by establishing the existence of a Lipschitz transport map.

\begin{theorem} \label{thm:mixtures}
	Let $\nu$ be a probability measure on $\RR^d$ with $ \support{\nu}\leq R$ and consider $\mu = \gamma_d\star\nu$. Then, for the map $\kmt:\RR^d \to \RR^d$, which satisfies $\kmt_*\gamma_d = \mu$, 
	$$\|\diff\kmt(x)\|_{\mathrm{op}} \leq e^{\frac{R^2}{8}},$$
		for almost every $x \in \RR^d$. Further, the example $d=1$ and
$\nu =\frac{1}{2}\delta_{-\frac{R}{2}}+\frac{1}{2}\delta_{\frac{R}{2}}$
shows that the Lipschitz bound is sharp. 
\end{theorem}

As mentioned above, the proofs of Theorems \ref{thm:lc} and \ref{thm:mixtures} follow from the analysis of Kim and Milman \cite{kim2012generalization}. The main result of \cite{kim2012generalization} is a generalization of Caffarelli's result that establishes Lipschitz properties of $\kmt$, under an appropriate symmetry assumption. We shall extend the analysis to the classes of measures considered in Theorems \ref{thm:lc}  and \ref{thm:mixtures}. A similar, but in some sense orthogonal to this work, extension was recently performed by Klartag and Putterman \cite[Section 3]{klartag2021spectral} where the authors considered transportation from $\mu$ to $\mu \star \gamma_d$. We also mention the concurrent work of Neeman in \cite{neeman2022lipschitz}, which, using a similar method to one presented here, studied Lipschitz properties of bounded perturbations of the Gaussians, generalizing \cite{colombo2017lipschitz}. In the broader context, a similar map was recently used in \cite{ambrosio2022quadratic}.

Both of the results presented above deal with Lipschitz transport maps that push-forward the standard Gaussian. As discussed, and as we shall demonstrate, the existence of such maps is important for applications. However, one could also ask the reverse question: for which measures $\mu$ do we have $\gamma_d = \varphi_*\mu$, with $\varphi$ Lipschitz?

To answer this question we introduce the class of $\conv$-semi-log-convex measures, as measures $\mu$ on $\RR^d$, which satisfy,
$$-\nabla^2\log \left(\frac{d\mu}{dx}(x)\right) \preceq \conv\Id_d,$$
for some $\beta > 0$. 
It follows from smoothness and the definition that $\mathrm{supp}(\mu)=\RR^d$ 
(which is why $\beta > 0$). In some sense, this is a complementary notion to being $\cc$-log-concave.
Our next result makes this intuition precise.
\begin{theorem} \label{thm:reverse}
	Let $\beta > 0$ and let $\mu$ be a $\beta$-semi-log-convex probability measure on $\RR^d$. Then, for the inverse map $(\kmt)^{-1}:\RR^d \to \RR^d$, which satisfies $(\kmt)^{-1}_*\mu = \gamma_d$, 
	$$\|\diff(\kmt)^{-1}(x)\|_{\mathrm{op}} \leq \sqrt{\conv},$$
	for almost every $x \in \RR^d$.
\end{theorem}
Let us remark that the same question was previously addressed in \cite[Theorem 2.2]{kolesnikov2011mass}, which expanded upon Caffarelli's original proof, and obtained the same Lipschitz bounds, for $(\optt)^{-1}$. Thus, Theorem \ref{thm:reverse} gives a more complete picture by proving the analogous result for the map $(\kmt)^{-1}$.

\subsection{Applications}
As mentioned in the previous section, for some applications it is essential that the domain and image of the transport map coincide. Here we review such applications and state several new implications of Theorems \ref{thm:lc}  and \ref{thm:mixtures}. To keep the statements concise, we will not cover applications that could be obtained by previous results, as in \cite{cordero2002some,mikulincer2021brownian,milman2018spectral}.

\subsubsection*{Eigenvalues comparisons}
A measure $\mu$ is said to satisfy a Poincar\'e inequality if, for some constant $C_\mathrm{p}(\mu) \geq 0$ and every test function $g$,
$$\var_\mu (g) \leq C_\mathrm{p}(\mu)\int\limits_{\RR^d} \|\nabla g\|^2d\mu.$$
We implicitly assume that, when it exists, $C_\mathrm{p}(\mu)$ denotes the optimal constant.
According to the Gaussian Poincar\'e inequality \cite{bakry2013analysis}, $C_\mathrm{p}(\gamma_d) = 1$. 
If $\mu = \varphi_*\gamma_d$ and $\varphi$ is $L$-Lipschitz, this immediately implies $C_\mathrm{p}(\mu) \leq L^2$. Indeed,
\begin{equation} \label{eq:poinacre}
\var_{\mu}(g)=\var_{\gamma_d}(g\circ \varphi)\le \int\limits_{\RR^d} \|\nabla(g\circ \varphi)\|^2d\gamma_d\le  \int\limits_{\RR^d} \|\diff\varphi\|_{\mathrm{op}}^2\,(\|\nabla g\|\circ \varphi)^2d\gamma_d\le L^2\int\limits_{\RR^d} \|\nabla g\|^2d\mu.
\end{equation}
Note that the same argument works even if $\varphi$ is a map between spaces of different dimensions. However, for certain generalizations of the Poincar\'e inequality, as we now explain, it turns out that it is beneficial for the domain of $\varphi$ to be the same as the
domain of $\mu$.
If $\frac{d\mu}{dx} = e^{-V}$ and we define the weighted Laplacian
$\mathcal{L}_\mu = \Delta - \langle\nabla, \nabla V\rangle$, then $C_\mathrm{p}(\mu)$ corresponds to the inverse of the first non-zero eigenvalue of $\mathcal{L}_\mu$. In \cite[Theorem 1.7]{milman2018spectral}, E. Milman showed that a similar argument to \eqref{eq:poinacre} works for higher order eigenvalues of $\mathcal{L}_\mu$ and $\mathcal{L}_{\gamma_d}$. 

Since for $\mathcal{L}_{\gamma_d}$ the multiplicities of the eigenvalues grow with the dimension $d$, the full power of Milman's argument requires that $\varphi$ is a map from $\RR^d$ to $\RR^d$. Thus, by considering the map $\kmt$ from Theorems \ref{thm:lc} and \ref{thm:mixtures} and applying Milman's contraction principle, we immediately obtain:

\begin{corollary}
	Let $\mu$ be a probability measure on $\RR^d$ and let $\lambda_i(\mathcal{L}_\mu)$ (resp. $\lambda_i(\mathcal{L}_{\gamma_d})$) stand for the $i^{\mathrm{th}}$ eigenvalue of $\mathcal{L}_\mu$ (resp. $\mathcal{L}_{\gamma_d}$). Then,
		\begin{enumerate}
		\item If $\mu$ is $\cc$-log-concave, $\s :=\support{\mu}$, and $\cc\s^2 < 4$,
		$$\frac{1}{\bdlcs} \lambda_i(\mathcal{L}_{\gamma_d}) \leq \lambda_i(\mathcal{L}_{\mu}).$$
		\item If $\mu = \gamma_d\star\nu$ and $ \support{\nu}\leq R$, then 
		$$\frac{1}{e^{\frac{R^2}{4}}} \lambda_i(\mathcal{L}_{\gamma_d}) \leq \lambda_i(\mathcal{L}_{\mu}).$$
	\end{enumerate}
\end{corollary}

\subsubsection*{Dimensional functional inequalities\footnote{A previous version of the paper erroneously claimed the validity of a dimensional log-Sobolev inequality, which we removed from this version.}}

Another direction for improving and generalizing the Poincar\'e inequality goes through dimensional functional inequalities. One such example is the dimensional weighted Poincar\'e inequality which appears in \cite[Corrolary 5.6]{bonnefont2016spectral}, according to which,
\begin{equation} \label{eq:dimpoinc}
	\var_{\gamma_d}(g) \leq \frac{d(d+3)}{d-1}\int\limits_{\RR^d}\frac{\|\nabla g(x)\|^2}{1+\|x\|^2}d\gamma_d(x).
\end{equation}
For certain test functions, this is a strict improvement of the Gaussian Poincar\'e inequality. 
When the target measure is symmetric, we can adapt the argument in \eqref{eq:poinacre}, and obtain:

\begin{corollary}
	Let $\mu$ be a symmetric probability measure on $\RR^d$. Then, for any test function $g:\RR^d \to \RR$,
	\begin{enumerate}
		\item If $\mu$ is $\cc$-log-concave, $\s :=\support{\mu}$, and $\cc\s^2 < 4$,
		$$\var_{\mu}(g) \leq \frac{d(d+3)}{d-1}		e^{1 - \cc\frac{\s^2}{4}} \frac{\s^2}{4}
\int\limits_{\RR^d}\frac{\|\nabla g(x)\|^2}{1+				e^{ \cc\frac{\s^2}{4}-1} \frac{4}{\s^2}
\|x\|^2}d\mu(x).$$

		\item If $\mu = \gamma_d\star\nu$ and $ \support{\nu}\leq R$,
		$$\var_{\mu}(g) \leq \frac{d(d+3)}{d-1}e^{\frac{R^2}{4}}\int\limits_{\RR^d}\frac{\|\nabla g(x)\|^2}{1+e^{-\frac{R^2}{4}}\|x\|^2}d\mu(x).$$
	\end{enumerate}
\end{corollary}

\begin{proof}
	Suppose that $\mu = \varphi_*\gamma_d$ where $\varphi:\RR^d \to \RR^d$ is $L$-Lipschitz and satisfies $\varphi(0) = 0$. Then, by \eqref{eq:dimpoinc},
	$$\var_{\mu}(g)=\var_{\gamma_d}(g\circ \varphi)\le \frac{d(d+3)}{d-1}\int\limits_{\RR^d} \frac{\|\nabla(g\circ \varphi(x))\|^2}{1+\|x\|^2}d\gamma_d\le \frac{d(d+3)L^2}{d-1} \int\limits_{\RR^d} \frac{(\|\nabla g\|\circ \varphi(x))^2}{1+\|x\|^2}d\gamma_d.$$
	To handle the integral on the right hand side, we invoke the disintegration theorem \cite[Theorems 1 and 2]{hoffmann1971existence} to decompose $\gamma_d$ along the fibers of $\varphi$ in the following way: there exists a family of probability measures $\{\gamma_x\}_{x\in \RR^d}$, such that $\mathrm{supp}(\gamma_x)\subset \varphi^{-1}(\{x\})$, and satisfies
	$$\int\limits_{\RR^d} h(x)d\gamma_d(x) = \int\limits_{\RR^d}\int\limits_{\varphi^{-1}(\{x\})}h(y)d\gamma_x(y)d\mu(x),$$
	for every test function $h$. Hence, taking $h(x) = \frac{(\|\nabla g\| \circ \varphi(x))^2}{1+\|x\|^2}$,
	\begin{align*}
		\int\limits_{\RR^d}& \frac{(\|\nabla g\|\circ \varphi(x))^2}{1+\|x\|^2}d\gamma_d(x) = \int\limits_{\RR^d} \int\limits_{\varphi^{-1}(\{x\})} \frac{(\|\nabla g\|\circ \varphi(y))^2}{1+\|y\|^2}d\gamma_x(y)d\mu(x)\\
		&= \int\limits_{\RR^d} \int\limits_{\varphi^{-1}(\{x\})} \frac{\|\nabla g(x)\|^2}{1+\|y\|^2}d\gamma_x(y)d\mu(x) \leq \int\limits_{\RR^d} \int\limits_{\varphi^{-1}(\{x\})} \frac{\|\nabla g(x)\|^2}{1+L^{-2}\|x\|^2}d\gamma_x(y)d\mu(x)\\
		&=\int\limits_{\RR^d}\frac{\|\nabla g(x)\|^2}{1+L^{-2}\|x\|^2}\left(\int\limits_{\varphi^{-1}(\{x\})} d\gamma_x(y)\right)d\mu(x) = \int\limits_{\RR^d}\frac{\|\nabla g(x)\|^2}{1+L^{-2}\|x\|^2}d\mu(x) 
	\end{align*}
where in the inequality we have used the estimate $\|y\| \geq \frac{1}{L}
\|x\|$ for any $y$ such that $\vphi(y) = x$. Indeed, by assumption,
$\vphi(0) = 0$ and $\vphi$ is L-Lipschitz, which immediately yields $\|\vphi(y)\| \leq L\|y\|.$

Finally, when $\mu$ is symmetric, our transport map, $\varphi := \kmt$, will turn out to be odd and, hence, satisfies $\kmt(0) = 0$ (see Remark \ref{rmk:symmetry}). The result follows by combining the above calculations with Theorems \ref{thm:lc} and \ref{thm:mixtures}.
\end{proof}

\subsubsection*{Majorization}
For an absolutely continuous measure $\mu$, define its distribution function by
$$F_{\mu}(\lambda) = \mathrm{Vol}\left(\left\{x: \frac{d\mu}{dx}(x)\geq \lambda\right\}\right).$$
We say that $\mu$ majorizes $\eta$, denoted as $\eta \prec \mu$, if for every $t\in \RR$,
$$\int\limits_t^\infty F_\eta(\lambda)d\lambda \leq \int\limits_t^\infty F_\mu(\lambda)d\lambda.$$
In \cite[Lemma 1.4]{melbourne2021transport}, the following assertion is proven: If $\mu = \varphi_*\eta$ for some $\varphi:\RR^d \to \RR^d$, and $|\det(\diff\varphi(x))|\leq 1$ for every $x \in \RR^d$, then $\eta \prec \mu$.

We use the singular value decomposition to deduce the identity $|\det(\diff\varphi(x))| = \prod\limits_{i=1}^d\sigma_i(\diff\varphi(x))$, where $\sigma_i(\diff\varphi(x))$ stands for the $i^{\mathrm{th}}$ singular value of $\diff\varphi(x)$. So, we have the implication,
$$\|\diff\varphi(x)\|_{\mathrm{op}} \leq 1 \implies |\det(\diff\varphi(x))| \leq 1.$$
By using Theorems \ref{thm:lc} and \ref{thm:mixtures} we can find regimes of parameters where $\kmt$ is
$1$-Lipschitz as required by the computation above. For log-concave measures it is enough to have a sufficiently bounded support, while for Gaussian mixtures one needs to both re-scale the variance and bound the support of the mixing measure.
With this in mind, we get the following corollary:

\begin{corollary} \label{cor:major}
	Let $\mu$ be a probability measure on $\RR^d$. 
	\begin{enumerate}
		\item If $\mu$ is $\cc$-log-concave, $\s :=\support{\mu}$, $\cc\s^2 < 1$, and $\bdlc \leq 1$, then,
		$$\gamma_d \prec \mu.$$
		\item If $\mu = \gamma_d^a\star \nu$, where $\gamma_d^a$ stands for the Gaussian measure with covariance $a\Id_d$, and $\sqrt{a}e^{\frac{\support{\nu}^2}{8a}} \leq 1$ then,
		$$\gamma_d \prec \mu.$$
	\end{enumerate}
\end{corollary}

\begin{proof}
	For the first part, the condition $\bdlc\leq 1$, along with Theorem \ref{thm:lc}, ensures that the transport map $\kmt$ is $1$-Lipschitz. The claim follows from \cite[Lemma 1.4]{melbourne2021transport}.
	
	For the second part, let $a > 0$ and $X \sim \gamma^a_d \star \nu$, where $\support{\nu} = R$. Then, $\frac{1}{\sqrt{a}}X \sim \gamma_d\star \tilde{\nu}$, and $\support{\tilde{\nu}} \leq \frac{R}{\sqrt{a}}$.
	Let $\kmt$ be the $e^{\frac{R^2}{8a}}$-Lipschitz map, from Theorem \ref{thm:mixtures}, that transports $\gamma_d$ to $\gamma_d\star\tilde{\nu}.$
	The above argument shows that $\sqrt{a}\kmt$ transports $\gamma_d$ to $\gamma^a_d \star \nu$ and the map is $\sqrt{a}e^{\frac{R^2}{8a}}$-Lipschitz. Thus, if $\sqrt{a}e^{\frac{R^2}{8a}} \leq 1$, there exists a $1$-Lipschitz transport map, which implies the result.
\end{proof}

\paragraph{Acknowledgments.} We wish to thank Max Fathi, Larry Guth, Emanuel Milman, and Ramon van Handel
 for several enlightening comments and suggestions. We also thank the anonymous referee for carefully reading the paper and providing many helpful comments that improved this manuscript.  We thank Sasha Kolesnikov for identifying a mistake in the proof of our claimed dimensional log-Sobolev inequality  (Corollary 5 in a previous version of this paper), which we removed. ChatGPT 5.5 Pro improved our key bounds in Lemma 4 and showed the optimality of the Gaussian mixture result. This material is based upon work supported by the National Science Foundation under Award Number 2002022.

\section{Proofs} \label{sec:proof} 
\subsection{Preliminaries} \label{sec:prelim} 
Before proving the main results, we briefly recall the construction of the transport map from \cite{kim2012generalization, otto2000generalization}. We take an informal approach and provide a rigorous statement at the end of the section.

Let $(Q_t)_{t\geq 0}$ stand for the Ornstein--Uhlenbeck semi-group, acting on functions $g:\RR^d \to \RR$ by,
$$Q_tg(x) =\int\limits_{\RR^d}g(e^{-t}x + \sqrt{1-e^{-2t}}y)d\gamma_d(y).$$
For sufficiently integrable $g$, we have, for almost every $x \in \RR^d$,
$$Q_0g(x) = g(x) \text{ and } \lim\limits_{t \to \infty} Q_tg(x) =  \EE_{\gamma_d}[g].$$ 
Now, fix $\mu$, a measure on $\RR^d$, with $\dens(x):=\frac{d\mu}{d\gamma_d}(x)$, and consider the measure-valued path $\mu_t := (Q_t\dens)\gamma_d$. We have $\mu_0 = \mu$ and, for well-behaved measures, we also have $\mu_t \xrightarrow{t \to \infty} \gamma_d.$
Thus, there exists a time-dependent vector field $V_t$, for which the continuity equation holds (see \cite[Chapter 8]{villani2003topics} and \cite[Section 4.1.2]{santambrogio2015optimal}):
$$\frac{d}{dt}\mu_t + \nabla \cdot (V_t\mu_t) = 0.$$
In other words, by differentiating under the integral sign, for any test function $g$,
$$\int\limits_{\RR^d}g\left(\frac{d}{dt}Q_t\dens\right)d\gamma_d = \int\limits_{\RR^d}\langle \nabla g  ,V_t\rangle (Q_t\dens)d\gamma_d.$$

We now turn to computing $V_t$. Observe that, by the definition of $Q_t$,
$$\frac{d}{dt}Q_t\dens(x) = \Delta Q_t\dens(x) - \langle x, \nabla Q_t\dens(x)\rangle.$$
Hence, integrating by parts with respect to the standard Gaussian shows,
$$\int\limits_{\RR^d}g\left(\frac{d}{dt}Q_t\dens\right)d\gamma_d = -\int\limits_{\RR^d}\langle \nabla g, \nabla Q_t\dens\rangle d\gamma_d,$$
whence it follows that $V_t = -\frac{\nabla Q_t\dens}{Q_t\dens} = -\nabla\log Q_t\dens.$
Now consider the maps $\{S_t\}_{t \geq 0}$, obtained as the solution to the differential equation
\begin{equation} \label{eq:differntial}
	\frac{d}{dt}S_t(x) =V_t(S_t(x)), \ \ \ S_0(x) = x.
\end{equation}
The map $S_t$ turns out to be a diffeomorphism that transports $\mu_0$ to $\mu_t$ and we denote $T_t : = S_t^{-1}$, which transports $\mu_t$ to $\mu_0$. We define the transport maps $T$ and $S$ as the limits 
$$T:=\lim\limits_{t \to \infty}T_t,\ \ S:=\lim\limits_{t \to \infty}S_t,$$
in which case, we have $T_*\gamma_d =  \mu$ and $S_*\mu = \gamma_d$. These are our transport maps 
$$\kmt:= T \quad\text{and}\quad (\kmt)^{-1} := S.$$
\begin{remark} \label{rmk:symmetry}
	It is clear that if $f(x) = f(-x)$, then $V_t$ and, consequently, $S_t$ (see the discussion following \cite[Lemma 3.1]{kim2012generalization}) are odd functions. Hence, if the target measure is symmetric, $T(0) = 0$.
\end{remark}
The above arguments are heuristic and require a rigorous justification (as in \cite[Section 3]{kim2012generalization}). 
For the sake of completeness, below, in Lemma \ref{lem:conditions}, we prove sufficient conditions for the existence of the diffeomorphisms $\{S_t\}_{t \geq 0}$, $\{T_t\}_{t \geq 0}$ and for the existence of the transport maps $S$ and $T$. 

We shall require the following approximation lemma, adapted from \cite[Lemma 2.1]{neeman2022lipschitz} (a generalization of \cite[Lemma 3.2]{kim2012generalization}), which we shall repeatedly use.
\begin{lemma} \label{lem:approx}
	Let $\eta$ and $\eta'$ be two probability measures on $\RR^d$, and let $\{\eta_k\}_{k \geq 0}$, $\{\eta'_k\}_{k \geq 0}$ be two sequences of probability measures which converge to $\eta$ and $\eta'$ in distribution. Suppose that for every $k$ there exists an $L_k$-Lipschitz map $\varphi_k$ with $(\varphi_k)_*\eta_k=\eta_k'$. Then, if $L:=\limsup\limits_{k \to \infty} L_k < \infty$, there exists an $L$-Lipschitz map $\varphi$ with $\varphi_*\eta=\eta'$. Moreover, by passing to a sub-sequence, we have that for $\eta$-almost every $x$,
	$$\lim\limits_{k\to \infty} \varphi_k(x) = \varphi(x).$$
\end{lemma}
\begin{proof}
	Under the assumptions of the lemma, the existence of the limiting map $\varphi$ is assured by the proof of \cite[Lemma 2.1]{neeman2022lipschitz}. We are left with showing that $\varphi$ is $L$-Lipschitz. Let $r > 0$, and observe that, since $\limsup L_k < \infty$, there exists a sub-sequence, still denoted $\vphi_k$, such that, for every $k \geq 0$, $\varphi_k$ is $(L+r)$-Lipschitz. It follows from \cite[Lemma 2.1]{neeman2022lipschitz} that $\vphi$ is $(L+r)$-Lipschitz. Since $r$ is arbitrary the proof is complete.
\end{proof}
We are now ready to state our main technical lemma.
\begin{lemma} \label{lem:conditions}
	 Assume that $\mu$ has a smooth density.
	\begin{itemize}
		\item Suppose that, for every $t \geq 0$, there exists $a_t < \infty $ such that,
		\begin{equation} \label{eq:lip}
			\sup\limits_{s\in [0,t]}\|\diff V_s\|_{\mathrm{op}} \leq a_t.
		\end{equation}
		Then, there exists a solution, $\{S_t\}_{t \geq 0}$, to \eqref{eq:differntial}, which is a diffeomorphism, for every $t \geq 0.$
		\item As $t \to \infty$, $\mu_t$ converges weakly to $\gamma_d$. 
		\item Suppose \eqref{eq:lip} holds, and that, for every $t \geq 0$, $T_t$ (resp. $S_t$) is $L_t$-Lipschitz. Then, if $L:= \limsup\limits_{t\to\infty}L_t < \infty$, the map $T$ (resp. $S$) is well-defined and $T$ (resp. $S$) is $L$-Lipschitz. 
	\end{itemize}
\end{lemma}
\begin{proof}
	Combining the assumption on the smoothness of $\frac{d\mu}{dx}$ with \eqref{eq:lip} gives that, for every $T < \infty$, $V$ is a smooth, spatially Lipschitz, function on $[0,T] \times \RR^d$ . Thus, by the Picard–Lindel\"of theorem, \cite[Theorem 3.1]{o1997existence}, there exists a unique global smooth (see \cite[Chapter 1, Theorem 3.3]{hale1980ordinary} and the subsequent discussion) solution $S_t$ to \eqref{eq:differntial}. By inverting the flow, one may see that the maps $S_t$ are invertible. Indeed, for fixed $t > 0$, consider, for $0\leq s \leq t$,
	$$\frac{d}{ds}T_{t,s}(x)= -V_{t-s}(T_{t,s}(x)), \ \ \ T_{t,0}(x) = x.$$
	Then, $S_t^{-1} := T_t := T_{t,t}$, which establishes the first item.
	
	For the second item, note that the Ornstein--Uhlenbeck process is ergodic (see, for example, \cite[Lemma 3.2]{kim2012generalization}) and, hence, $$\lim\limits_{t \to \infty} \|Q_tf - \EE_{\gamma_d}[f]\|_{L_1(\gamma_d)} = \lim\limits_{t \to \infty} \|Q_tf - 1\|_{L_1(\gamma_d)} = 0.$$
	Thus, $\mu_t$ converges to $\gamma_d$ in total variation, implying weak convergence.
	
	To see the third item, note that the first item establishes the existence of maps $S_t$ which satisfy, $(S_t)_*\mu = \mu_t$, \cite[Section 4.1.2]{santambrogio2015optimal}. The second item shows that, as $t \to \infty$, we may approximate $\gamma_d$ by $\mu_t$. These conditions allow us to invoke Lemma \ref{lem:approx}, which shows that there exists a sequence $t_k \xrightarrow{k \to \infty} 1$, such that, for $\mu$-almost every $x$, 
	$$S(x):= \lim\limits_{k \to \infty}S_{t_k}(x),$$
	is well-defined and such that $S_*\mu = \gamma_d$. Since $S_t$ is invertible, for every $t \geq 0$, the same argument, applied to $T_t$, shows the existence of $T$. 
	
	Finally, let us address the Lipschitz constants of $S$ and $T$. We shall prove the claim for $S$; the proof for $T$ is identical. The previous argument shows that there exists a null set $E \subset \mathrm{supp}(\mu)$, such that, for every $z \in \mathrm{supp}(\mu) \setminus E$, $\lim\limits_{k \to \infty}S_{t_k}(z)$ exists. So, for any $x,y \in \mathrm{supp}(\mu) \setminus E$,
	\begin{align*}
		\|S(x) - S(y)\| = \lim\limits_{k \to \infty}\|S_{t_k}(x) - S_{t_k}(y)\| \leq \limsup_{k \to \infty} L_{t_k}\|x - y\| \leq L\|x-y\|.
	\end{align*}
	This shows $\|S(x) - S(y)\| \leq L\|x-y\|$, $\mu$-almost everywhere, which finishes the proof.
\end{proof}
We shall also require the following lemma, which explains how to deduce global Lipschitz bounds from estimates on the derivatives of the vector fields $V_t$.
\begin{lemma} \label{lem:boundtolipschitz}
	Let the above notation prevail and assume that $\mu$ has a smooth density. For every $t \geq 0$, let $\theta^{\mathrm{max}}_t,\theta^{\mathrm{min}}_t$ be such that $$\theta^{\mathrm{max}}_t \geq \lambda_{\mathrm{max}}\left(-\diff V_t(x)\right) \geq \lambda_{\mathrm{min}}\left(-\diff V_t(x)\right)\geq \theta^{\mathrm{min}}_t,$$ for almost every $x \in \RR^d$. Then, 	
	\begin{enumerate}
		\item  The Lipschitz constant of $S$ is at most $\exp\left(-\int\limits_0^\infty\theta^{\mathrm{min}}_tdt\right)$.
		\item  The Lipschitz constant of $T$ is at most $\exp\left(\int\limits_0^\infty\theta^{\mathrm{max}}_tdt\right)$. 
	\end{enumerate} 
\end{lemma}
\begin{proof}
	We begin with the first item. For every $t \geq 0$, we will show that
	\begin{equation} \label{eq:expansive}
		\|S_t(x)-S_t(y)\| \leq \exp\left(-\int\limits_0^t\theta^{\mathrm{min}}_sds\right)\|x-y\| \text{ for every } x,y \in \RR^d.
	\end{equation}
	The desired result will be obtained by taking $t \to \infty$ and invoking Item 3 of Lemma \ref{lem:conditions}.
	
	Towards \eqref{eq:expansive}, it will suffice to show that, for every unit vector $w \in \RR^d$, 
	$$\|\diff S_t(x) w\|\leq \exp\left(-\int\limits_0^t\theta^\mathrm{min}_sds\right).$$
	Fix $x,w \in \RR^d$ with $\|w\|=1$, and define the function $\alpha_w(t) := \diff S_t(x) w$.
	To understand the evolution of $\|\alpha_w(t)\|$, recall that $S_t$ satisfies the differential equation in \eqref{eq:differntial}. Thus,
	\begin{align*}
		\frac{d}{dt}\|\alpha_w(t)\| &= \frac{1}{\|\alpha_w(t)\|}\alpha_w(t)^\T\cdot \frac{d}{dt}\alpha_w(t)=\frac{1}{\|\alpha_w(t)\|}w^\T\diff S_t(x)^\T\diff V_t(S_t(x))\diff S_t(x) w\\
		&\leq -\theta_t^\mathrm{min}\frac{1}{\|\alpha_w(t)\|}w^\T\diff S_t(x)^\T\diff S_t(x) w = -\theta^\mathrm{min}_t \|\diff S_t(x) w\| = -\theta_t^\mathrm{min} \|\alpha_w(t)\|.
	\end{align*}
	Since $\|\alpha_w(0)\| =1$, from Gronwall's inequality we deduce,
	$$\|\diff S_t(x) w\| = \|\alpha_w(t)\| \leq \exp\left(-\int\limits_0^t\theta_s^\mathrm{min}ds\right).$$
	Thus, \eqref{eq:expansive} is established, as required.\\
	
	The proof of the second part is similar, but this time we will need to show, for every unit vector $w \in \RR^d$, 
	$$\|\diff S_t(x) w\|\geq \exp\left(-\int\limits_0^t\theta^\mathrm{max}_sds\right).$$
	Indeed, this would imply $\diff S_t(x) \diff S_t(x)^\T \succeq \exp\left(-2\int\limits_0^t\theta^\mathrm{max}_sds\right)\Id_d$. Since $S_t$ is a diffeomorphism, and $T_t = S_t^{-1}$, by the inverse function theorem, the local expansiveness of $S_t$ implies
	$$\diff T_t(x)\diff T_t(x)^\T \preceq \exp\left(2\int\limits_0^t\theta^\mathrm{max}_sds\right)\mathrm{I}_d.$$
	So, for almost every $x \in \RR^d$, $\|\diff T_t(x)\|_{\mathrm{op}} \leq \exp\left(\int\limits_0^t\theta^\mathrm{max}_sds\right)$, and the claim is proven by, again, invoking Item 3 in Lemma \ref{lem:conditions}.
	
	Let $\alpha_w(t)$ be as above. Then,
	\begin{align*}
		\frac{d}{dt}\|\alpha_w(t)\| &= \frac{1}{\|\alpha_w(t)\|}\alpha_w(t)^\T\cdot \frac{d}{dt}\alpha_w(t)=\frac{1}{\|\alpha_w(t)\|}w^\T\diff S_t(x)^\T\diff V_t(S_t(x))\diff S_t(x) w\\
		&\geq -\theta^\mathrm{max}_t\frac{1}{\|\alpha_w(t)\|}w^\T\diff S_t(x)^\T\diff S_t(x) w = -\theta^\mathrm{max}_t \|\diff S_t(x) w\| = -\theta^\mathrm{max}_t \|\alpha_w(t)\|.
	\end{align*}
	As before, Gronwall's inequality implies
	$$\|\diff S_t(x) w\| = \|\alpha_w(t)\| \geq \exp\left(-\int\limits_0^t\theta^\mathrm{max}_sds\right),$$
	which concludes the proof.
\end{proof}
\subsection{Lipschitz properties of transportation along heat flows}
\subsubsection*{Transportation from the Gaussian}
Our proofs of Theorems \ref{thm:lc} and \ref{thm:mixtures} go through bounding the derivative, $\nabla V_t = -\nabla^2\log Q_t\dens$, of the vector field constructed above, and then applying Lemma \ref{lem:boundtolipschitz}. Our main technical tools are uniform estimates on $\nabla^2\log Q_t\dens$, when the measures satisfy some combination of convexity and boundedness assumptions.

\begin{lemma} 
	\label{lem:nablav}
	Let $\mu=\dens \gamma_d$ and let $\s:=\support{\mu}$. Then, for almost every $x \in \RR^d$,
	$$-\diff V_t(x) \succeq -\frac{e^{-2t}}{1-e^{-2t}}\Id_d.$$
	Furthermore, 
	\begin{enumerate}
		\item For every $t\geq 0$,
		\[
		-\diff V_t(x) 	\preceq e^{-2t}\left(\frac{1}{4}\frac{\s^2}{(1-e^{-2t})^2}-\frac{1}{1-e^{-2t}}\right)\Id_d.
		\]
		\item Let $\cc\in \RR$ and suppose that $\mu$ is $\cc$-log-concave. Then,
		\[
		-\diff V_t(x) 	\preceq e^{-2t}\frac{1-\cc}{\cc(1-e^{-2t})+e^{-2t}},
		\]
		where the inequality holds for any $t \geq 0$ when $\cc \geq 0$, and for $t\in \left[0, \frac{1}{2}\log\left(\frac{\cc-1}{\cc}\right)\right]$ if $\cc <0$.
		\item If $\mu:=\gamma_d\star \nu$, with $\support{\nu}\leq R$, then, for $t \geq 0$,
		\[
		-\diff V_t(x) 	\preceq e^{-2t}\frac{R^2}{4}\Id_d.
		\]
	\end{enumerate}
\end{lemma}
\begin{proof}
We follow the argument of \cite[Lemma 3.3]{mikulincer2021brownian}. Let $(P_t)_{t \in [0,1]}$ stand for the heat semi-group, related to $Q_t$ by $Q_t\dens(x) = P_{1-e^{-2t}}\dens(e^{-t}x)$. In particular,
\begin{equation}
\label{eq:V_to_Hessian}
-\diff V_t(x) = \nabla^2 \log Q_t\dens(x) = e^{-2t}\nabla^2 \log P_{1-e^{-2t}}\dens(e^{-t}x).
\end{equation}
The last term can be written as a covariance (cf. \cite[Equation (3.3)]{mikulincer2021brownian}),
\begin{equation}
\label{eq:cov_id}
 e^{-2t}\nabla^2 \log P_{1-e^{-2t}}\dens(e^{-t}x)= \frac{e^{-2t}}{(1-e^{-2t})^2}\cov(p^{x,1-e^{-2t}})-\frac{e^{-2t}}{1-e^{-2t}}\Id_d,
\end{equation}
where
\begin{equation*}
d p^{x,t}(y)=\frac{f(y)\phi^{x,t}(y)}{P_tf(x)}d y,
\end{equation*}
with $\phi^{x,t}$ being the density of a $d$-dimensional Gaussian distribution with mean $x$ and covariance $t\Id_d$. Thus, our goal 
is to upper bound $\cov(p^{x,1-e^{-2t}})$, towards which we will use Popoviciu's inequality on variances. For item 1 Popoviciu's inequality gives 
\begin{equation*}
\cov(p^{x,1-e^{-2t}})\preceq \frac{\s^2}{4}\Id_d,
\end{equation*}
which together with \eqref{eq:V_to_Hessian}-\eqref{eq:cov_id} yields the result. For item 2 we note that if $\mu$ is $\cc$-log-concave then $p^{x,1-e^{-2t}}$  is $\left(\cc+\frac{e^{-2t}}{1-e^{-2t}}\right)$-log-concave, at which point the Brascamp-Lieb inequality can be applied; see the proof of \cite[Lemma 3.3(2)]{mikulincer2021brownian}. For item 3 we follows the proof of  \cite[Lemma 3.3(3)]{mikulincer2021brownian} so that
	\begin{equation*}
\cov(p^{x,1-e^{-2t}})=(1-e^{-2t})\Id_d+(1-e^{-2t})^2\cov(Z),
\end{equation*}
where $Z$ is a random vector supported on the support of $\nu$. By Popoviciu's inequality,
	\begin{equation*}
\cov(p^{x,1-e^{-2t}})\preceq \left[(1-e^{-2t})\Id_d+(1-e^{-2t})^2\frac{R^2}{4}\right]\Id_d,
\end{equation*}	
which together with \eqref{eq:V_to_Hessian}-\eqref{eq:cov_id} proves the result.
\end{proof}

By integrating Lemma \ref{lem:nablav} and plugging the result into
Lemma \ref{lem:boundtolipschitz} we can now prove Theorems \ref{thm:lc} and \ref{thm:mixtures}. We begin with the proof of
Theorem \ref{thm:mixtures}, which is easier.

\begin{proof}[Proof of Theorem \ref{thm:mixtures}]
Recall that $\kmt$ is the transport map $T$, constructed in Section \ref{sec:prelim}. Remark that the conditions of Lemma \ref{lem:conditions} are satisfied for the measures we consider: Lemma \ref{lem:nablav} ensures that \eqref{eq:lip} holds and $\mu$ has a smooth density.

If $\mu:=\gamma_d\star \nu$, and $\nu$ is such that $ \support{\nu}\leq R$, then, by Lemma \ref{lem:nablav}, we may take $\theta^\mathrm{max}_t =  e^{-2t}\frac{R^2}{4}$ in Lemma \ref{lem:boundtolipschitz}. Compute 
\begin{align*}
	\int\limits_0^\infty \theta^\mathrm{max}_tdt = \frac{R^2}{8}.
\end{align*}
Thus, $\kmt$ is Lipschitz with constant $e^{\frac{R^2}{8}}$. It remains to show the sharpness of the bound when $d=1$ and $\nu =\frac{1}{2}\delta_{-\frac{R}{2}}+\frac{1}{2}\delta_{\frac{R}{2}}$. When $d=1$, the map $\kmt$ is the same as the Brenier map  \cite[Section 6]{kim2012generalization}, and hence $\kmt$  is odd as $\gamma_1$ and $\mu$ are symmetric. In particular, $\kmt(0)=0$. By the change of variables formula,
\begin{equation}
\label{eq:change_variable_mix}
\left|\frac{d}{dx}\kmt(0)\right|=\frac{\gamma_1(0)}{\frac{1}{2}\gamma_1(\kmt(0)-\frac{R}{2})+\frac{1}{2}\gamma_1(\kmt(0)+\frac{R}{2})}=e^{\frac{R^2}{8}},
\end{equation}
as claimed. 
\end{proof}

The proof of Theorem \ref{thm:lc} is similar, but the calculations involved are more tedious, even if elementary.

\begin{proof}[Proof of Theorem \ref{thm:lc}]
We begin by assuming that $\mu$ has a smooth density, and handle the general case later with an approximation argument. Thus, as in the proof of Theorem \ref{thm:mixtures}, the conditions of Lemma \ref{lem:conditions} are satisfied, and we recall that $\kmt$ is the transport map $T$. The first item of the Theorem is covered by \cite[Theorem 1.1]{kim2012generalization} (the authors actually prove it for $\cc = 1$; the general case follows by a re-scaling argument), so we may assume $\cc\s^2 < 4$. 
Set $t_0 = \frac{1}{2} \log\Big(\frac{A (\cc - 1) - 1}{\cc A - 1}\Big)$ with $A=\frac{\s^2}{4}$. By optimizing over the first and second estimates in Lemma \ref{lem:nablav} we define,
$$\theta^\mathrm{max}_t = \begin{cases}
	 \frac{e^{-2t}(1-\cc)}{\cc(1-e^{-2t})+e^{-2t}}& \text{if }t \in [0,t_0]\\
	e^{-2t}\left(\frac{A}{(1-e^{-2t})^2}-\frac{1}{1-e^{-2t}}\right)& \text{if }t > t_0
\end{cases}.$$
Remark that when $\cc < 0$, $t_0 < \frac{1}{2}\log\left(\frac{\cc-1}{\cc}\right)$, so the second bound of Lemma \ref{lem:nablav} remains valid in this case.\\

We compute,
\begin{align*}
	\int\limits_0^\infty \theta^\mathrm{max}_tdt &= \int\limits_0^{t_0} \theta^\mathrm{max}_tdt + \int\limits_{t_0}^{\infty} \theta^\mathrm{max}_tdt\\
	&= \int\limits_0^{t_0}\frac{e^{-2t}(1-\cc)}{\cc(1-e^{-2t})+e^{-2t}}dt + \int\limits_{t_0}^{\infty}e^{-2t}\left(\frac{A}{(1-e^{-2t})^2}-\frac{1}{1-e^{-2t}}\right)dt\\
	&= -\frac{1}{2}\log(\cc(1-e^{-2t}) + e^{-2t})\Bigg\vert_{0}^{t_0} +\frac{1}{2}\left(-\frac{A}{1-e^{-2t}}-\log(1-e^{-2t})\right)\Bigg\vert_{t_0}^{\infty}\\
	&= \frac{1}{2}\log\left(1 - A(\cc-1)\right) + \frac{1 - \cc A}{2} +\frac{1}{2}\log(A) - \frac{1}{2}\log(1-A(\cc-1))\\
	&= \frac{1 - \cc A}{2} +\frac{1}{2}\log(A).
\end{align*}
Recalling $A=\frac{\s^2}{4}$ we conclude by Lemma \ref{lem:boundtolipschitz} that the Lipschitz constant of $\kmt$ is at most
$$\exp\left(\int\limits_0^\infty \theta^\mathrm{max}_tdt\right) = e^{\frac{1 - \cc\frac{\s^2}{4}}{2}} \frac{\s}{2}.$$
\end{proof}

\subsubsection*{Transportation to the Gaussian} 
To prove Theorem \ref{thm:reverse} we will need an analogue of Lemma \ref{lem:nablav} with bounds in the other direction. This is done in the following lemma which shows that the evolution of log-convex functions along the heat flow is dominated by the evolution of Gaussian functions. The proof of the lemma is similar to the
proof that strongly log-concave measures are preserved under convolution, \cite[Theorem 3.7(b)]{saumard2014log}. The only difference between the proofs is that the use of the Pr\'ekopa-Leindler inequality is replaced by the fact that a mixture of log-convex
functions is log-convex.
\begin{lemma}[Semi-log-convexity under the heat flow] \label{lem:convexbound}
	Let $d\mu = fd\gamma$ be a $\conv$-semi-log-convex probability measure on $\RR^d$. Then, for almost every $x$,
	$$-\nabla V_t(x) \succeq \frac{e^{-2t}\left(1 - \conv\right)}{(1-e^{-2t})\left(\conv -1\right) +1}\mathrm{I}_d.$$
\end{lemma}
\begin{proof}
	We let $(P_t)_{t \geq 0}$ stand for the heat semi-group, defined by
	$$P_tf(x) = \int\limits_{\RR^d} f(x + \sqrt{t}y)d\gamma_d(y).$$
	 Since $-\nabla V_t(x) = \nabla^2 \log Q_t\dens(x) = e^{-2t}\nabla ^2\log P_{1-e^{-2t}}\dens(e^{-t}x)$, it will be enough to prove,
	 \begin{equation} \label{eq:heatbound}
	 \nabla^2\log P_tf(x) \succeq \frac{\left(1 - \conv\right)}{t\left(\conv -1\right) +1}\mathrm{I}_d.
	 \end{equation}
	 We first establish the claim in the special case when $f(x) := \psi_\beta(x) \propto e^{-\frac{1}{2}\left(\conv - 1\right)\|x\|^2}$, where the symbol $\propto$ signifies equality up to a constant which does not depend on $x$, which corresponds to $\mu = \mathcal{N}(0, \frac{1}{\conv}\mathrm{I}_d).$ 
	This case is facilitated by the fact that $P_t$ acts on $f$ by convolving it with a Gaussian kernel. The result follows since a convolution of Gaussians is a Gaussian and since $\nabla^2\log$ applied to a Gaussian yields the covariance matrix.
	To elucidate what comes next, we provide below the full calculation.
	
	For convenience denote $\conv_t =  \left(t\left(\beta-1\right) + 1 \right)$, and compute,
	\begin{align*}
		P_t\psi_\beta(x) &\propto \int\limits_{\RR^d} e^{- \frac{1}{2}\left(\conv-1\right)\|x + \sqrt{t}y\|^2}e^{-\frac{\|y\|^2}{2}}dy \\
		&= \int\limits_{\RR^d} \exp\left(-\frac{1}{2}\left(\left(\conv-1\right)\|x\|^2 + 2\left(\conv-1\right)\sqrt{t}\langle x, y\rangle + \left(t\left(\conv-1\right) + 1 \right)\|y\|^2\right)\right)dy\\
		&= \int\limits_{\RR^d} \exp\left(-\frac{\conv_t}{2}\left(\frac{\conv-1}{\conv_t}\|x\|^2 + 2\sqrt{t}\frac{\conv-1}{\conv_t}\langle x, y\rangle + \|y\|^2\right)\right)dy\\
		&= \exp\left(-\frac{\conv_t}{2}\left(\frac{\conv-1}{\conv_t}\left(1 - t\frac{\conv-1}{\conv_t}\right)\right)\|x\|^2\right)\int \exp\left(-\frac{\conv_t}{2}\left\|\sqrt{t}\frac{\conv-1}{\conv_t}x+y\right\|^2\right)dy.
	\end{align*}
	The integrand in the last line is proportional to the density of a Gaussian. Hence, the value of the integral does not depend on $x$, and
	\begin{align*}
		P_t\psi_\beta(x)&\propto \exp\left(-\frac{\conv_t}{2}\left(\frac{\conv-1}{\conv_t}\left(1 - t\frac{\conv-1}{\conv_t}\right)\right)\|x\|^2\right)\\
		&= \exp\left(-\frac{1}{2}\left(\left(\conv-1\right)\left(1 - t\frac{\conv-1}{\conv_t}\right)\right)\|x\|^2\right)\\
		&= \exp\left(-\frac{1}{2}\left(\frac{\conv-1}{t\left(\conv-1\right)+1}\|x\|^2\right)\right).
	\end{align*}
	So, 
	\begin{align} \label{eq:specialgauss}
		\nabla^2 \log P_t\psi_\beta(x) &= \frac{\left(1 - \conv\right)}{t\left(\conv -1\right) +1}\mathrm{I}_d,
	\end{align}
	which gives equality in \eqref{eq:heatbound}.
	
	For the general case, the log-convexity assumption means that we can write $\frac{d\mu}{dx} =  e^{V(x)- \conv\frac{\|x\|^2}{2}},$ for a convex function $V$. Hence,
	$f(x) \propto e^{V(x)- \frac{1}{2}(\conv-1)\|x\|^2}$. With analogous calculations to the ones made above, we get,
	\begin{align*}
		P_tf(x) &\propto \int\limits_{\RR^d} e^{V(x +\sqrt{t}y)- \frac{1}{2}\left(\conv-1\right)\|x + \sqrt{t}y\|^2}e^{-\frac{y^2}{2}}dy\\
		&= \exp\left(-\frac{1}{2}\left(\frac{\conv-1}{t\left(\conv-1\right)+1}\|x\|^2\right)\right) \int\limits_{\RR^d} \exp\left(V(x +\sqrt{t}y)-\frac{\conv_t}{2}\left\|\sqrt{t}\frac{\conv-1}{\conv_t}x+y\right\|^2\right)dy\\
		&\propto P_t\psi_\beta(x)\int\limits_{\RR^d} \exp\left(V(x +\sqrt{t}y)-\frac{\conv_t}{2}\left\|\sqrt{t}\frac{\conv-1}{\conv_t}x+y\right\|^2\right)dy.
	\end{align*}
	Write $H_t(x) := \int\limits_{\RR^d} \exp\left(V(x +\sqrt{t}y)-\frac{\conv_t}{2}\left\|\sqrt{t}\frac{\conv-1}{\conv_t}x+y\right\|^2\right)dy$ and observe by \eqref{eq:specialgauss},
	\begin{equation} \label{eq:heatconvex}
	\nabla^2\log P_t f(x) = \nabla^2 \log P_t\psi_\beta(x) + \nabla^2\log(H_t(x)) =  \frac{\left(1 - \conv\right)}{t\left(\conv -1\right) +1}\mathrm{I}_d + \nabla^2\log(H_t(x)).
	\end{equation}
	To finish the proof we will show that $\nabla^2\log(H_t(x)) \succeq 0$, or , equivalently, that $H_t$ is log-convex. By applying a linear change of variables, we can re-write $H_t$ as,
	$$H_t(x) = \int\limits_{\RR^d} \exp\left(V\left(\left(1-t\frac{\conv - 1}{\conv_t}\right)x + \sqrt{t}y\right)\right)e^{-\frac{\conv_t\|y\|^2}{2}}dy.$$
	As $V$ is convex, for every $t \geq 0$ and $y\in \RR^d$, the function $x \mapsto V\left(\left(1-t\frac{\conv - 1}{\conv_t}\right)x + \sqrt{t}y\right)$ is convex. So, $H_t(x)$ is a mixture of log-convex functions. Since a mixture of log-convex functions is also log-convex (see \cite[Chapter 16.B]{marshall2011inequalities}), the proof is complete.
\end{proof}
 
 We now prove Theorem \ref{thm:reverse}.
 \begin{proof}[Proof of Theorem \ref{thm:reverse}]
 	Recall that  $(\kmt)^{-1}$ is the transport map $S$, constructed in Section \ref{sec:prelim}. Again, we begin by assuming that $\mu$ has a smooth density, and one may verify that the conditions of Lemma \ref{lem:conditions} are satisfied, which makes $S$ well-defined.  
 	
 	Let $\theta_t^\mathrm{min}= e^{-2t}\frac{\left(1 - \conv\right)}{(1-e^{-2t})\left(\conv -1\right) +1}$. Combining Lemma \ref{lem:convexbound} with Lemma \ref{lem:boundtolipschitz} shows that $S$ is $\exp\left(\int_0^\infty-\theta_t^\mathrm{min}dt\right)$-Lipschitz.
 	Compute,
 	\begin{align*}
 		\int_0^\infty-\theta_t^\mathrm{min}dt = \int_0^\infty -e^{-2t}\frac{\left(1 - \conv\right)}{(1-e^{-2t})\left(\conv -1\right) +1}dt &= \frac{1}{2} \log\left((1 - e^{-2 t})(\conv - 1) +1\right)\Big\vert_0^\infty\\
 		&=\frac{\log(\conv)}{2}.
 	\end{align*}
	Hence, $S$ is $\exp\left(\frac{\log(\conv)}{2}\right) = \sqrt{\conv}$-Lipschitz.
	
	To finish the proof, we shall construct a family $\{\mu_\eps\}_{\eps > 0}$ of $\beta_\eps$-log-convex measures which converge to $\mu$ in distribution as $\eps \to 0$, and such that
	$$ \lim\limits_{\eps \to 0}\beta_\eps = \beta.$$
	The claim then follows by invoking Lemma \ref{lem:approx}.
	
	Let $\gamma_{d,\eps}$ stand for the $d$-dimensional Gaussian measure with covariance $\eps \mathrm{I}_d$, and set $\mu_\eps = \mu \star \gamma_{d,\eps}$. It is clear that, as $\eps \to 0$, $\mu_\eps$ converges to $\mu$ in distribution. Moreover, if we replace $f$ by $\frac{d\mu}{dx} $, in \eqref{eq:heatconvex}, we see that $\mu_\eps$ is $\beta_\eps$-log-convex, with,
	$$\beta_\eps:= \frac{\beta}{\eps\beta +1} \xrightarrow{\eps \to 0} \beta.$$
 \end{proof}
\bibliographystyle{plain}
\bibliography{bib}{}

\begin{thebibliography}{10}

\bibitem{ambrosio2022quadratic}
Luigi Ambrosio, Michael Goldman, and Dario Trevisan.
\newblock On the quadratic random matching problem in two-dimensional domains.
\newblock {\em Electron. J. Probab.}, 27:--, 2022.

\bibitem{bakry2013analysis}
Dominique Bakry, Ivan Gentil, and Michel Ledoux.
\newblock {\em Analysis and geometry of Markov diffusion operators}, volume
  348.
\newblock Springer Science \& Business Media, 2013.

\bibitem{bardet2018functional}
Jean-Baptiste Bardet, Natha\"{e}l Gozlan, Florent Malrieu, and Pierre-Andr\'{e}
  Zitt.
\newblock Functional inequalities for {G}aussian convolutions of compactly
  supported measures: explicit bounds and dimension dependence.
\newblock {\em Bernoulli}, 24(1):333--353, 2018.

\bibitem{bolley2018dimensional}
Fran\c{c}ois Bolley, Ivan Gentil, and Arnaud Guillin.
\newblock Dimensional improvements of the logarithmic {S}obolev, {T}alagrand
  and {B}rascamp-{L}ieb inequalities.
\newblock {\em Ann. Probab.}, 46(1):261--301, 2018.

\bibitem{bonnefont2016spectral}
Michel Bonnefont, Ald\'{e}ric Joulin, and Yutao Ma.
\newblock Spectral gap for spherically symmetric log-concave probability
  measures, and beyond.
\newblock {\em J. Funct. Anal.}, 270(7):2456--2482, 2016.

\bibitem{brenier1991polar}
Yann Brenier.
\newblock Polar factorization and monotone rearrangement of vector-valued
  functions.
\newblock {\em Comm. Pure Appl. Math.}, 44(4):375--417, 1991.

\bibitem{caffarelli2000monotonicity}
Luis~A. Caffarelli.
\newblock Monotonicity properties of optimal transportation and the {FKG} and
  related inequalities.
\newblock {\em Comm. Math. Phys.}, 214(3):547--563, 2000.

\bibitem{chen2021dimension}
Hong-Bin Chen, Sinho Chewi, and Jonathan Niles-Weed.
\newblock Dimension-free log-{S}obolev inequalities for mixture distributions.
\newblock {\em J. Funct. Anal.}, 281(11):Paper No. 109236, 17, 2021.

\bibitem{colombo2017lipschitz}
Maria Colombo, Alessio Figalli, and Yash Jhaveri.
\newblock Lipschitz changes of variables between perturbations of log-concave
  measures.
\newblock {\em Ann. Sc. Norm. Super. Pisa Cl. Sci. (5)}, 17(4):1491--1519,
  2017.

\bibitem{cordero2002some}
Dario Cordero-Erausquin.
\newblock Some applications of mass transport to {G}aussian-type inequalities.
\newblock {\em Archive for rational mechanics and analysis}, 161(3):257--269,
  2002.

\bibitem{frieze1999sobolev}
Alan Frieze and Ravi Kannan.
\newblock Log-{S}obolev inequalities and sampling from log-concave
  distributions.
\newblock {\em Ann. Appl. Probab.}, 9(1):14--26, 1999.

\bibitem{hale1980ordinary}
Jack~K. Hale.
\newblock {\em Ordinary differential equations}.
\newblock Robert E. Krieger Publishing Co., Inc., Huntington, N.Y., second
  edition, 1980.

\bibitem{hoffmann1971existence}
J{\o}rgen Hoffmann-J{\o}rgensen.
\newblock Existence of conditional probabilities.
\newblock {\em Mathematica Scandinavica}, 28(2):257--264, 1971.

\bibitem{kannan1995isoperimetric}
Ravi Kannan, L{\'a}szl{\'o} Lov{\'a}sz, and Mikl{\'o}s Simonovits.
\newblock Isoperimetric problems for convex bodies and a localization lemma.
\newblock {\em Discrete \& Computational Geometry}, 13(3-4):541--559, 1995.

\bibitem{kim2012generalization}
Young-Heon Kim and Emanuel Milman.
\newblock A generalization of {C}affarelli's contraction theorem via (reverse)
  heat flow.
\newblock {\em Math. Ann.}, 354(3):827--862, 2012.

\bibitem{klartag2021spectral}
Bo'az Klartag and Eli Putterman.
\newblock Spectral monotonicity under gaussian convolution.
\newblock {\em arXiv preprint arXiv:2107.09496}, 2021.

\bibitem{kolesnikov2011mass}
Alexander~V Kolesnikov.
\newblock Mass transportation and contractions.
\newblock {\em arXiv preprint arXiv:1103.1479}, 2011.

\bibitem{marshall2011inequalities}
Albert~W. Marshall, Ingram Olkin, and Barry~C. Arnold.
\newblock {\em Inequalities: theory of majorization and its applications}.
\newblock Springer Series in Statistics. Springer, New York, second edition,
  2011.

\bibitem{melbourne2021transport}
James Melbourne and Cyril Roberto.
\newblock Transport-majorization to analytic and geometric inequalities.
\newblock {\em arXiv preprint arXiv:2110.03641}, 2021.

\bibitem{mikulincer2021brownian}
Dan Mikulincer and Yair Shenfeld.
\newblock The {B}rownian transport map.
\newblock {\em arXiv preprint arXiv:2111.11521}, 2021.

\bibitem{milman2018spectral}
Emanuel Milman.
\newblock Spectral estimates, contractions and hypercontractivity.
\newblock {\em J. Spectr. Theory}, 8(2):669--714, 2018.

\bibitem{neeman2022lipschitz}
Joe Neeman.
\newblock Lipschitz changes of variables via heat flow.
\newblock {\em arXiv preprint arXiv:2201.03403}, 2022.

\bibitem{o1997existence}
Donal O'Regan.
\newblock {\em Existence theory for nonlinear ordinary differential equations},
  volume 398 of {\em Mathematics and its Applications}.
\newblock Kluwer Academic Publishers Group, Dordrecht, 1997.

\bibitem{otto2000generalization}
F.~Otto and C.~Villani.
\newblock Generalization of an inequality by {T}alagrand and links with the
  logarithmic {S}obolev inequality.
\newblock {\em J. Funct. Anal.}, 173(2):361--400, 2000.

\bibitem{payne1960optimal}
L.~E. Payne and H.~F. Weinberger.
\newblock An optimal {P}oincar\'{e} inequality for convex domains.
\newblock {\em Arch. Rational Mech. Anal.}, 5:286--292 (1960), 1960.

\bibitem{santambrogio2015optimal}
Filippo Santambrogio.
\newblock {\em Optimal transport for applied mathematicians}, volume~87 of {\em
  Progress in Nonlinear Differential Equations and their Applications}.
\newblock Birkh\"{a}user/Springer, Cham, 2015.
\newblock Calculus of variations, PDEs, and modeling.

\bibitem{saumard2014log}
Adrien Saumard and Jon~A. Wellner.
\newblock Log-concavity and strong log-concavity: a review.
\newblock {\em Stat. Surv.}, 8:45--114, 2014.

\bibitem{tanana2021comparison}
Anastasiya Tanana.
\newblock Comparison of transport map generated by heat flow interpolation and
  the optimal transport {B}renier map.
\newblock {\em Commun. Contemp. Math.}, 23(6):Paper No. 2050025, 7, 2021.

\bibitem{villani2003topics}
C\'{e}dric Villani.
\newblock {\em Topics in optimal transportation}, volume~58 of {\em Graduate
  Studies in Mathematics}.
\newblock American Mathematical Society, Providence, RI, 2003.

\bibitem{wang2016functional}
Feng-Yu Wang and Jian Wang.
\newblock Functional inequalities for convolution probability measures.
\newblock {\em Ann. Inst. Henri Poincar\'{e} Probab. Stat.}, 52(2):898--914,
  2016.

\end{thebibliography}
\end{document}